\documentclass{amsart}
\usepackage[utf8]{inputenc}
\usepackage{latexsym}
\usepackage{amsfonts}
\usepackage[all]{xy}
\usepackage{color}
\usepackage{amssymb, mathrsfs, amsfonts, amsmath}
\usepackage{amsbsy}
\usepackage{tikz}
\usepackage{array,tabularx,float, multicol,multirow,graphicx}
\usetikzlibrary{calc,intersections,arrows.meta,through,patterns}
\newtheorem{theorem}{Theorem}[section]
\newtheorem{thm}{Theorem}[section]
 \newtheorem{cor}[thm]{Corollary} 
 
 \newtheorem{remark}[thm]{Remark}
 \newtheorem{prop}[thm]{Proposition}

 \newtheorem{definition}[thm]{Definition}
 \newcommand{\grad}{\nabla\,}
 
 \newcommand\raisepunct[1]{\,\mathpunct{\raisebox{0.5ex}{#1}}}
\usepackage{graphicx} 

\title{Concentration of mean exit times}

\author[G. P. Bessa]{G. Pacelli Bessa}
\address{Department of Mathematics, Universidade Federal do Cear\'{a}-UFC, 60.440-900, Fortaleza, Brazil}
\email{bessa@mat.ufc.br}

\author[V. Gimeno]{Vicent Gimeno i Garcia}
\address{Department of Mathematics, Universitat Jaume I-IMAC,   E-12071, 
Castell\'{o}, Spain}
\email{gimenov@uji.es}

\author[V. Palmer]{Vicente Palmer}
\address{Department of Mathematics, Universitat Jaume I-INIT,   E-12071, 
Castell\'{o}, Spain}
\email{palmer@mat.uji.es}

\begin{document}
\thanks{Research partially supported by the Research grant  PID2020\- -115930GA\- -100 funded by MCIN/ AEI /10.13039/50110001103 and  by the project AICO/2023/035 funded by Conselleria d’Educació, Cultura, Universitats i Ocupació}

\begin{abstract}
     The mean exit time function defined on the $\delta$-tube 
     around any equator $\mathbb{S}^{n-1} \subseteq \mathbb{S}^{n}$ of the sphere $\mathbb{S}^{n}$, ($0<\delta<\pi/2$), goes to infinity with the dimension, so that when we consider a Brownian particle that begins its motion at one equator of the sphere, this particle will remain near this equator for an almost infinite amount of time when the dimension of the sphere goes to infinity. On the other hand, if the Brownian particle begins its motion at the North pole, then this particle will leave quickly, when the dimension of the sphere goes to infinity, any geodesic ball with radius $\delta <\pi/2$, centered at this point. Namely, the mean exit time function defined on the equatorial tubes presents a kind of {\em concentration} phenomenon or {\em fat equator} effect, as it has been described in the book \cite{MS}. 
     
     Moreover, the same concentration phenomenon occurs when we consider this mean exit time function defined on tubes around closed and minimal hypersurfaces of a compact Riemannian $n$-manifold $M$ with Ricci curvature bounded from below, ${\rm Ric}_{M}\geq (n-1)$. Namely, a Brownian particle that begins its random movement around a closed embedded minimal hypersurface of a compact $n$-manifold $M$ with  ${\rm Ric}_{M}\geq (n-1)$ will wanders arbitrarily close to the hypersurface for a time that approaches infinity as the dimension of the ambient manifold does so as well. 
\end{abstract}
\maketitle
\section{Introduction}

Let $\{X_t\}$ be the standard Brownian motion on a complete Riemannian manifold $M$ and $\Omega \subset M$ be a nonempty subset. The expected value $\mathbb{E}(\tau^{x}_{\Omega})$ of the variable  \begin{equation}\tau_{\Omega}^{x}=\inf\{t>0 \colon   X_t\not \in \Omega; X_0=x\in \Omega\}\in [0, \infty]\label{eq1}\end{equation}  defines a (possibly extended) function  $E_{\Omega} \colon \Omega \to [0, \infty]$,  by  $x\to E_{\Omega}(x)=\mathbb{E}(\tau^{x}_{\Omega})$,  called the mean exit time function of $\Omega$. If $\Omega$ is bounded with $\partial \Omega \neq \emptyset$ then $E_{\Omega}<\infty$.
The mean exit time is a fundamental concept intertwining  Riemannian geometry and stochastic processes, describing the expected time for a Brownian particle  to leave a given domain for the first time. For instance, the mean exit time is related to the first eigenvalue,   
\cite{banuellos-carroll94}, \cite{banuelos23}, \cite{DM}, \cite{HMP12}, \cite{HMP16}, to isoperimetric inequalities \cite{banuellos-carroll11}, \cite{Cadeddu15}, \cite{palmer99} and to many other geometric aspects of the manifolds as one can  see in \cite{gimeno-palmer21}, \cite{gual-maso05}  \cite{HMP09},  \cite{karp-pinsky}, \cite{markvorsen89}, \cite{miquel-palmer92}
\cite{palmer98} and  references therein. 

In this article, and in a situation very similar to that contemplated in the paper \cite{GiPa}, we will consider the mean exit time function of a tubular neighborhood of the sphere's equator $\mathbb{S}^{n-1}$ and show that we have a concentration phenomenon similar to the volume concentration on the sphere, as it has been described in the book \cite{MS}, (see too Section 9 in \cite{Gromov}).  To put our result in perspective, while emphasizying the more geometric aspect of the phenomenon of concentration of measure in the sphere, let us consider  $\mathbb{S}^{n}\subset \mathbb{R}^{n+1}$ be the unit sphere equipped with the metric of curvature $+1$ and consider the $\delta$-tube $T_{\delta}(\mathbb{S}^{n-1})$ around the equator  $\mathbb{S}^{n-1}\subset\mathbb{S}^n$, with $0<\delta<\pi/2$. It is well known that Levy's isoperimetric inequality in the sphere implies the inequality, (see \cite{MS}),   $${\rm vol}(T_{\delta}(\mathbb{S}^{n-1}))\geq (1-2e^{-(n-1)\delta^2/2}){\rm vol}(\mathbb{S}^{n}).$$

This inequality implies, in its turn, that when the dimension $n$ of the sphere increases, then the volume of the tube ${\rm vol}(T_{\delta}(\mathbb{S}^{n-1}))$ approaches the volume of the entire sphere $\mathbb{S}^{n}$, for any $0<\delta<\pi/2$.




\begin{figure}    
\begin{center}

\begin{tikzpicture}[scale=1.5,>=stealth]
		\def\u{1}
		\def\v{2}
		\def\z{0.3}
		
		\coordinate (p) at (0,0);
		\coordinate (p0) at ($ (p)+(\u,0) $);
		\coordinate (p1) at ($ (p)+(0,\u cm) $);
         \coordinate (p4) at ($ (p)-(0,\u cm) $);
		
		\def\k{45}
		\coordinate (p2) at ($ (p1)+(0,\u cm) $);
		\coordinate (p3) at ($ (p)+(\k:\v cm) $);
		
		\path[draw,name path=disco, line width=1pt] (p) circle (\u cm);
		\path[draw=black!10,name path=disco, fill=black!10] (p1) circle (1pt)node[above] {$ N $};

		\path[draw,name path=arco2,line width=1pt] (p0) arc (360:180:\u cm and 0.3*\u cm);
		\path[draw,name path=arco1,dashed] (p0) arc (0:180:\u cm and 0.3*\u cm);

		\begin{scope}
		\clip[] (p) circle (\u cm);
		\path[draw,name path=arco3] (1,0.3) arc (360:180:\u cm and 0.3*\u cm);
		\path[draw,name path=arco4] (1,-0.3) arc (360:180:\u cm and 0.3*\u cm);
		\path[draw,name path=arco5,fill=white!70!black] (1,-0.3) arc (360:180:\u cm and 0.3*\u cm);
		\fill[white!70!black] (-1,-0.3) rectangle (1,0.3);
		\path[draw,name path=arco1,dashed] (1,0.2) arc (0:180:\u cm and 0.3*\u cm);
	    \end{scope}

	    \begin{scope}
	    	\clip[] (p) circle (\u cm);
	    	\path[draw,name path=arco3,fill=white] (1,0.3) arc (360:180:\u cm and 0.3*\u cm);
	    	\path[draw,name path=arco4] (1,-0.3) arc (360:180:\u cm and 0.3*\u cm);
	    	
	    \end{scope}
	    
	    \path[draw,name path=disco, line width=1pt] (p) circle (\u cm);
	    \path[draw,name path=arco2,line width=1pt] (p0) arc (360:180:\u cm and 0.3*\u cm);
	    \path[draw=black!10,name path=disco, fill=black!10] (p1) circle (1pt)node[above] {$ N $};
      \path[draw=black!10,name path=disco, fill=black!10] (p4) circle (1pt)node[below] {$ S$};
	    \path[draw,name path=arco1,dashed] (p0) arc (0:180:\u cm and 0.3*\u cm);

	    \node[] (t1) at (2,1) {$ B_{(\pi/2-\delta)}(p_N) $};
	    \node[] (t2) at (2,0.4) {$T_{\delta}(\mathbb{S}^{n-1})$};
        \node[] (t4) at (2,-1) {$ B_{(\pi/2-\delta)}(p_S) $};

	    \coordinate (p4) at (45:1.1cm);
	    \coordinate (p5) at (-10:1.1cm);
        \coordinate (p6) at (-45:1.1cm);
	    
	    \draw[->,shorten <=4pt] (t1) to[out=-90,in=45] (p4);
	    \draw[->,shorten <=4pt] (t2) to[out=-90,in=0] (p5);
        \draw[->,shorten <=4pt] (t4) to[out=-90,in=-45] (p6);

	\end{tikzpicture}
\end{center}
\caption{Tube $T_\delta(\mathbb{S}^{n-1})$ of radius $\delta$ around the 
 equator $\mathbb{S}^{n-1}$ and geodesic balls $B_{\pi/2-\delta}(p_N)$, $B_{\pi/2-\delta}(p_S)$ of radius $\pi/2-\delta$ centered at the north ans south pole respectively.}
    \label{figure}
\end{figure}
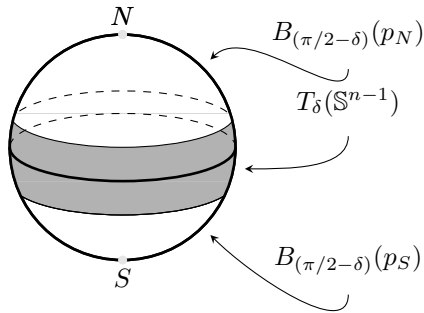
This geometric interpretation of the concentration of measure phenomenon suggests that the time that the standard Brownian motion $X_t\in \mathbb{S}^n$ takes to leave $T_{\delta}(\mathbb{S}^{n-1})$ concentrates as well, in the sense that the mean exit time of the first exit of $X_t$ of $T_{\delta}(\mathbb{S}^{n-1})$ tends to $+\infty$  as $n\to \infty$. In order to present our results, we are going to introduce some notation and terminology.

Let us consider $\mathbb{S}^{n}\subset \mathbb{R}^{n+1}$ be the $n$-sphere of curvature $1$. Let us fix the north pole $p_N=(0,\cdots,0, 1)$  and let us consider the geodesic ball $B^{\mathbb{S}^{n}}_{\pi/2}(p_N)$ of radius $\pi/2$ centered at $p_N$, see figure \ref{figure}. We shall consider $\mathbb{S}^{n-1}=\partial B^{\mathbb{S}^{n}}_{\pi/2}(p_N)$ the equator with respect to  $p_N$. Note that any point of the sphere can be placed in the north pole position using the accurate rotation, so the arguments and the results we present in the paper, which refer to quantities that are invariant under the action of isometries, apply for any point and any equator of the sphere. We are considering too the equator $ \mathbb{S}^{n-1}$ immersed in $\mathbb{S}^{n}$ as a totally geodesic submanifold, $\varphi: \mathbb{S}^{n-1} \rightarrow \mathbb{S}^{n}$.

Let $T_\delta(\mathbb{S}^{n-1})$ be the tube of radius $\delta$ around $\mathbb{S}^{n-1}$, in the ambient sphere $\mathbb{S}^{n}$, defined below, (see Section \ref{tubes}), using the distance function $s(x)={\rm dist}_{\mathbb{S}^{n}}( x\, , \mathbb{S}^{n-1}).$  Let $u_{\delta,n}\colon T_\delta(\mathbb{S}^{n-1}) \to [0, \infty)$ be the mean exit time function defined on  $T_\delta(\mathbb{S}^{n-1})$. It is well known that this function is radial, namely,  $u_{\delta}(x)=u_{\delta}(s(x))$ and satisfies  the following Poisson problem, see \cite{dynkin},

\begin{equation}\label{eqMET2}
    \left\{\begin{array}{rl}
         \triangle_{_{\mathbb{S}^{n}}} u_{\delta,n}+1=0& \!{\rm in}\,\,T_{\delta}(\mathbb{S}^{n-1})  \\
         u_{\delta,n}=0&\!\!{\rm on} \, \,\partial T_{\delta}(\mathbb{S}^{n-1}).
    \end{array}\right.
\end{equation}

The mean exit time function $v_{\delta,n}(x)=v_{\delta,n}(t(x))$ from the ball $B_{\pi/2 - \delta}(p_N)$ of radius $\pi/2 - \delta$ and center at the north pole $p_N$ is radial too, in the sense that depends on the distance to the north pole, $t(x)= {\rm dist}_{\mathbb{S}^{n}}( x\, , p_N)$ and satisfies the equation
\begin{equation}\label{eqMET3}
    \left\{\begin{array}{rl}
         \triangle_{_{\mathbb{S}^{n}}} v_{\delta,n}+1=0& \!{\rm in}\,\,B_{\pi/2 - \delta}(p_N)  \\
         v_{\delta,n}=0&\!\!{\rm on} \, \,\partial B_{\pi/2 - \delta}(p_N).
    \end{array}\right.
\end{equation}

Our first result confirms the suggestion that the mean exit time defined on the $\delta$-tube around the equator $T_{\delta}(\mathbb{S}^{n-1})$, $0<\delta<\pi/2$, goes to infinity with the dimension,  while goes to zero, with the dimension too, when we consider it defined on the geodesic balls. Namely, the mean exit time function presents a kind of {\em concentration} phenomenon when the dimension of the sphere goes to infinity which can be described in the following terms, depending on whether we are considering the solution of the problems \eqref{eqMET2} or \eqref{eqMET3}:

When we consider a Brownian particle that begins its motion at one equator of the sphere, this particle will remain near this equator for an almost infinite amount of time when the dimension of the sphere goes to infinity, which in its turn comes from the fact that the mean exit time function $u_{\delta,n}$ will be bounded from below by a one-variable function $F_{\delta,n}$ which goes to infinity in each value of its dominion when $n$ goes to infinity. 

On the other hand, if the Brownian particle begins its motion at the North pole, then this particle will leave quickly, when the dimension of the sphere goes to infinity, any geodesic ball centered at the North pole, with radius $\delta <\pi/2$, because the mean exit time function $v_{\delta,n}$ will be uniformly bounded from above by a constant depending on $\delta$ and $n$ which goes to zero when $n$ goes to infinity.
 
 All these considerations are formalized in the following statement.

\begin{theorem}\label{thmmain}
Let us consider the tube  $T_\delta(\mathbb{S}^{n-1}) \subseteq \mathbb{S}^{n}$, of radius $\delta \in (0,\pi/2)$, around the totally geodesic submanifold $\varphi: \mathbb{S}^{n-1} \rightarrow \mathbb{S}^{n}$, and the geodesic ball $B_{\pi/2-\delta}(p_N) \subseteq \mathbb{S}^{n}$ centered at the north pole $p_N$. Then, we have the following statements:
\begin{itemize}
\item[i)] The mean exit time function $u_{\delta,n}$ defined on the tube $T_{\delta}(\mathbb{S}^{n-1})$, is the radial solution of \eqref{eqMET2}, and it is bounded from below by the function $F_{\delta,n}\circ s$ as
$$
\begin{aligned}
u_{\delta,n}(x)=u_{\delta,n}(s(x))&\geq (F_{\delta,n}\circ s)(x)\,\,\,\forall x \in T_{\delta}(\mathbb{S}^{n-1})\raisepunct{.}
\end{aligned} $$
\noindent where $s(x)=\mathrm{dist}_{\mathbb{S}^{n}}(x,\mathbb{S}^{n-1})$ and  $F_{\delta,n}: (0,\delta) \rightarrow \mathbb{R}$ is the function defined as
$$F_{\delta,n}(r)=\frac{(\delta-r)}{\sin(\delta)}\cdot\frac{1-\cos^{n}(r)}{n\cos^{n-1}(r)}\raisepunct{.} $$
\noindent which satisfies  that
$$ \lim_{n\to \infty}F_{\delta,n}(r)=\infty \,\,\,\forall r \in (0,\delta)$$

\item[ii)] The mean exit time function $v_{\delta,n}$ from the ball  $B_{\pi/2-\delta}(p_N) \subseteq \mathbb{S}^{n}$ is the radial solution of \eqref{eqMET3}, and it is uniformly  bounded from above as $$v_{\delta,n}(y)=v_{\delta,n}(t(y))\leq G_{\delta,n}:= \frac{1-\cos(\pi/2-\delta)}{n\cos(\pi/2-\delta)}\cdot$$\label{thmprelim}
\noindent where $t(y)=\mathrm{dist}_{\mathbb{S}^{n}}(y,p_N)$, and we have that the uniform bound $G_{\delta,n}$ satisfies
 $$\lim_{n\to \infty}G_{\delta,n}=0\, .$$
 \end{itemize}
\end{theorem} 

The mean exit time function defined in a domain $\Omega \subseteq M$  is the first in a sequence of functions $\{u_0=1,u_{1,\Omega,n}=E_\Omega\;, u_{2,\Omega,n}\;, ....\}$  defined inductively as follows

\begin{equation}\label{poisson}
		\begin{split}
			\Delta^Mu_{1,\Omega,n}+1 & =0,\,\text{ on }\Omega,\\
			u_{1,\Omega,n}\lvert_{_{\partial \Omega}} & = 0,
		\end{split}
	\end{equation}
	
	\noindent and, for $k\geq 2$,
	
	\begin{equation}\label{poissonk}
		\begin{split}
			\Delta^Mu_{k,\Omega,n}+ku_{k-1,\Omega,n} & =  0,\,\, \text{on }\,\,\Omega,\\
			u_{k,\Omega,n}\lvert_{_{\partial \Omega}} & = 0.
		\end{split}
	\end{equation}

\noindent This sequence is the so-called {\em Poisson hierarchy for} $\Omega$,  see \cite{DLD}, \cite{Mc2}.

As occurs with the mean exit time function $u_{1,\Omega,n}=E_\Omega$, (which is the second element of the Poisson hierarchy, or, in fact, its first non-trivial element), it can be proved, (see \cite{Mc2} and \cite{Ha}), that, given the domain $\Omega \subseteq M$,  the $k^{\rm th}$-element of the Poisson hierarchy for $\Omega$ agrees with the  $k^{\rm th}$-exit moment of $\Omega$, $k\geq 1$, which is the expected value of the $k^{\rm th}$-power of  variable $\tau_{\Omega}$ given in \eqref{eq1}, i.e., 
$$u_{k,\Omega,n}(x)=\mathbb{E}[(\tau_{\Omega}^{x})^k]$$

When we restrict ourselves to our setting, in which we have $\Omega=T_{\delta}(\mathbb{S}^{n-1})$ or $\Omega=B_{\pi/2-\delta}(p_N)$, we will adopt the following notation for the elements of the Poisson hierarchy: in case $\Omega=T_{\delta}(\mathbb{S}^{n-1})$, then $u_{k,\Omega,n}=u_{\delta,k,n}(x)$ and in case $\Omega=B_{\pi/2-\delta}(p_N)$, then  $v_{k,\Omega,n}=v_{\delta,k,n}(x)$. 

In view of Theorem \eqref{thmmain},  it is natural to expect that the $k^{\rm th}$-exit moments, $u_{\delta,k,n}$ and $v_{\delta,k,n}$, experience the same concentration phenomenon that the mean exit time functions $u_{\delta,1,n}$ and $v_{\delta,1,n}$.
Indeed, this is the content of the following result.

\begin{cor}\label{thmmain2}  

Let us consider the tube $T_\delta(\mathbb{S}^{n-1}) \subseteq \mathbb{S}^{n}$, of radius $\delta \in (0,\pi/2)$, around the totally geodesic submanifold $\varphi: \mathbb{S}^{n-1} \rightarrow \mathbb{S}^{n}$, and the geodesic ball $B_{\pi/2-\delta}(p_N) \subseteq \mathbb{S}^{n}$ centered at the north pole $p_N$. 
\begin{itemize}
\item[i)] The $k^{\rm th}$-exit moment $u_{\delta,k,n}$ defined on the tube $T_{\delta}(\mathbb{S}^{n-1})$ is bounded from below by the  a function $F_{\delta,k,n}\circ s$ as
$$
\begin{aligned}
u_{\delta,k,n}(x)&=u_{\delta,k,n}(s(x))\geq (F_{\delta,k,n}\circ s)(x)\,\,\,\,\forall x \in T_{\delta}(\mathbb{S}^{n-1})\raisepunct{.}
\end{aligned} $$
\noindent where $s(x)=\mathrm{dist}_{\mathbb{S}^{n}}(x,\mathbb{S}^{n-1})$ and  $F_{\delta,k,n}: (0,\delta) \rightarrow \mathbb{R}$ is a function which satisfies  that
$$ \lim_{n\to \infty}F_{\delta,k,n}(r)=\infty \,\,\,\forall r \in (0,\delta), \,\,\,\forall k \in \mathbb{N} $$

\item[ii)] The $k^{\rm th}$-exit moment  $v_{\delta,k,n}$ defined on the ball  $B_{\pi/2-\delta}(p_N) \subseteq \mathbb{S}^{n}$  is uniformly  bounded from above as 
$$ v_{\delta,k,n}(s(x)) \leq k! \left(v_{\delta,1,n}(0)\right)^k $$ \noindent and we have that the unifom bound $k! \left(v_{\delta,1,n}(0)\right)^k $ satisfies
 $$\lim_{n\to \infty}k! \left(v_{\delta,1,n}(0)\right)^k =0\, .$$
 \end{itemize}
\end{cor}


Our second main result generalizes the concentration phenomenon satisfied by the mean exit time established in Theorem \ref{thmmain}, when we consider this function defined on tubes around closed and minimal hypersurfaces of a compact Riemannian $n$-manifold $M$ with Ricci curvature bounded from below, ${\rm Ric}_{M}\geq (n-1)$. 

Namely, we have obtained a comparison theorem for the mean exit time $w_{\delta,n}$ of the $\delta$-tube $T_\delta(\Sigma)$ around a compact embedded minimal hypersurface $\Sigma$ of  a compact Riemannian manifold $M$ with Ricci curvature ${\rm Ric}_{M}\geq (n-1)$ and the radial mean exit time $ u_{\delta,n}$ of the tube $T_{\delta}(\mathbb{S}^{n-1})\subset \mathbb{S}^{n}$,  transplanted to the tube $T_\delta(\Sigma)$  using the distance function to $\Sigma$, $s(x)={\rm dist}_{M}(x, \Sigma)$, by means the equality $$ \tilde{u}_{\delta,n}(x):=u_{\delta,n}(s(x))\,\,\forall x \in T_\delta(\Sigma)$$

In this setting, the tube radius $\delta$ is bounded above by 
${\rm min}\{\pi/2,C(\Sigma)\}$, where $C(\Sigma)$ is the minimal focal distance of $\Sigma$, defined in Section \ref{tubes}.
This comparison leads to the same conclusion as above about the behavior of a Brownian particle that begins its random movement around a closed embedded minimal hypersurface of a compact $n$-manifold $M$ with  ${\rm Ric}_{M}\geq (n-1)$: it wanders arbitrarily close to the hypersurface for a time that approaches infinity as the dimension of the manifold does so as well. We must remark at this point that the use of these techniques does not allow obtaining upper bounds for the mean exit time function defined on the geodesic balls of such a manifold.
    \begin{theorem}\label{thmmain3} Let $M$ be a compact Riemannian $n$-manifold with Ricci curvature ${\rm Ric}_{M}\geq (n-1)$ and $\varphi \colon \Sigma \to M^n$ be a closed embedded minimal hypersurface. Let  $T_\delta(\Sigma)$ and $T_\delta(\mathbb{S}^{n-1})$ be the $\delta$-tube of $\Sigma$ and $\partial B_{\pi/2}(p_N)$ respectively, with $0<\delta< {\rm min}\{\pi/2,C(\Sigma)\}$, (see Section \ref{tubes} for the definition of the minimal focal distance $C(\Sigma)$). The mean exit time $w_{\delta,n}(x)$ of $T_{\delta}(\Sigma)$ is bounded below as 
    \begin{equation}
    \begin{aligned}w_{\delta,n}(x)\geq \tilde{u}_{\delta,n}(s(x))&\geq (F_{\delta,n}\circ s)(x)\,\,\,\,\forall x \in T_{\delta}(\Sigma)\raisepunct{.}
\end{aligned}
\end{equation} 
    \noindent where $s(x)={\rm dist}_{M}(x, \Sigma)$, and $F_{\delta,n}$ is the function defined in Theorem \ref{thmmain}, which satisfies
    
    $$ \lim_{n\to \infty}F_{\delta,n}(r)=\infty, \,\,\,\forall r \in (0,\delta).$$
    Moreover, if  $w_{\delta,n}(x_0)= \tilde{u}_{\delta,n}(s(x_0))$ for some $x_0\in T_\delta(\Sigma)$, and some $n \in \mathbb{N}$ then
\begin{enumerate}
    \item Every tubular hypersurface  $\partial T_r(\Sigma)$ has constant mean curvature, $r\in (0, \delta)$.
    \item[]
 \item The volume of $\Sigma$ is bounded from above
 $$
{\rm vol}(\Sigma)\leq \frac{1}{1-2e^{-(n-1)\delta^{2}/2}}\cdot {\rm vol}(\mathbb{S}^{n-1}).
$$
 \end{enumerate}
\end{theorem}

\begin{remark}
As we have said before, Theorem \ref{thmmain3} tell us that the mean exit time function $w_{\delta,n}(x)$ becomes larger and larger as the dimension $n$ of the manifold $M$ increases. 
On the other hand, and with regard to the bound on the Ricci curvature, which is established in terms of the dimension $n$, we must take into account, when interpreting the behavior of the curvatures of the manifold when the dimension increases, that although the Ricci curvature increases with the dimension, the sectional curvatures can remain constant in any plane (as occurs with spheres or projective spaces), or in any case, bounded in any plane, from above and from below, by finite quantities. 
\end{remark}
\medskip

An analogous statement to that of Theorem \ref{thmmain3} is true for the $k^{\rm th}$-exit moments defined on $T_\delta(\Sigma)$, just as it was stablished in Corollary \ref{thmmain2} for the spherical model, as we can see in the following result. We remark here that, in order to get the comparison, we have transplanted the radial $k^{\rm th}$-exit moments defined on $T_\delta(\mathbb{S}^{n-1})$ to the tube $T_\delta(\Sigma)$ in the same way as we did in Theorem \ref{thmmain3}.

\begin{cor}\label{cor:moments}Let $M$ be a compact Riemannian $n$-manifold with Ricci curvature ${\rm Ric}_{M}\geq (n-1)$ and $\varphi \colon \Sigma \to M^n$ be a closed embedded minimal hypersurface. The  $k^{\rm th}$-exit moments $w_{\delta,k,n}$ and the transplanted $k^{\rm th}$-exit moments $\tilde{u}_{\delta,k,n}$ both defined in $T_\delta(\Sigma)$, with $0<\delta< {\rm min}\{\pi/2,C(\Sigma)\}$, satisfy 
\begin{equation}w_{\delta,k,n}(x)\geq \tilde{u}_{\delta,k,n}(s(x))\geq (F_{\delta,k,n}\circ s)(x),\,\,\,\,\forall x \in T_\delta(\Sigma).
\end{equation}
In particular, by Corollary \eqref{thmmain2} we have that $F_{\delta,k,n}: (0,\delta) \rightarrow \mathbb{R}$ is a function which satisfies  that
$$ \lim_{n\to \infty}F_{\delta,k,n}(r)=\infty, \,\,\,\forall r \in (0,\delta), \,\,\,\forall k \in \mathbb{N} $$ 
\end{cor}

The following result is a direct consequence of Theorem \ref{thmmain3}  and Grove-Shiohama sphere theorem.

\begin{cor}\label{teo:Shiohama}Let  $\varphi \colon \Sigma \to M^n$ be a closed embedded minimal hypersurface of a   compact Riemannian $n$-manifold $M$  with sectional curvatures  $K_M\geq 1$. The mean exit time $w_{\delta,n}(x)$ of $T_{\delta}(\Sigma)$, with $0<\delta< {\rm min}\{\pi/2,C(\Sigma)\}$, is bounded below as 
    \begin{equation}
    \begin{aligned}w_{\delta,n}(x)\geq \tilde{u}_{\delta,n}(s(x))&\geq (F_{\delta,n}\circ s)(x)\,\,\,\,\forall x \in T_{\delta}(\Sigma)\raisepunct{.}
.\label{eqcor}
\end{aligned}
\end{equation} 
    \noindent where $s(x)={\rm dist}_{M}(x, \Sigma)$, and     
    
    $$ \lim_{n\to \infty}F_{\delta,n}(r)=\infty, \,\,\,\forall r \in (0,\delta).$$

   \noindent  Moreover, if  $w_{\delta,n}(x_0)= \tilde{u}_{\delta,n}(s(x_0))$ for some $x_0\in T_\delta(\Sigma)$, and some $n \in \mathbb{N}$ and
   $$\frac{1}{2}\cdot\frac{1}{1-2e^{-(n-1)\delta^{2}/2}}\cdot {\rm vol}(\mathbb{S}^{n-1})<{\rm vol}(\Sigma)$$
   \noindent then $M$ is homeomorphic to $\mathbb{S}^n$.
\end{cor}
\bigskip

\section{Tubes around submanifolds}\label{tubes}
 In this section we will recall the geometry of tubes around submanifolds, following the book \cite{gray}. Let $\varphi \colon \!\Sigma\to M$ be an isometric embedding    of a complete Riemannian $m$-manifold $\Sigma$  into a Riemannian $n$-manifold $M$. Let us identify $\Sigma$ with the zero section  $s\colon \Sigma \to T\Sigma^{\perp}$  of the normal bundle. The exponential map $\exp^{\perp}\colon T\Sigma^{\perp}\to M$ is a diffeomorphism when restricted to an open neighborhood $\Omega_{\Sigma}$ of $ T\Sigma^{\perp}$ containing $\Sigma$. To describe this set $\Omega_{\Sigma}$ observe that for every $q\in \exp^{\perp}(\Omega_{\Sigma})$ there exists a unique minimal geodesic  $\gamma_{\xi}\colon [0, \ell]\to M $ joining  $p=\gamma_{\xi}(0)\in \Sigma$ to $\gamma_{\xi}(\ell)=q$ with  $\gamma_{\xi}'(0)=\xi\in  \big(T_{\gamma_{\xi}(0)}\Sigma\big)^{\perp}$. The geodesic $\gamma_{\xi}$  realize the distance $\ell={\rm dist}_{_M}(\Sigma, q)$ between $q$ and $\Sigma$.
\vspace{2mm}

Letting $c\colon ST\Sigma^{\perp}=\{(p,\xi)\in T\Sigma^{\perp}\colon \Vert \xi\Vert=1\}\to \mathbb{R}$ defined as $$c(p, \xi)=\sup_{t>0}\{t\colon {\rm dist}_{_M}(p, \gamma_{\xi} (t))=t\} $$ we have that $$ \exp^{\perp}\big(\Omega_{\Sigma}\big)= \{ \exp^{\perp}(p,s\xi)=\gamma_{\xi}(s),\,\,\, (p,\xi)\in ST\Sigma^{\perp}\,,\,0\leq s\leq c(p,\xi)\}.  $$ The minimal focal distance $C(\Sigma)$ of $\Sigma$ in $M$ is defined as $$C(\Sigma)=\inf\{c(p, \xi), (p, \xi) \in ST\Sigma^{\perp}\}.$$  
\begin{definition}
Let us consider $\varphi: \Sigma^m \longrightarrow M^n$ a submanifold of a Riemannian manifold $M$, and a number  $r\leq C(\Sigma)$. The {\em tube of radius} $r$ about $\Sigma$ in $M$ is the set
\begin{equation*}
  T_r(\Sigma):=\left\{ x \in M\, :\, \text{dist}_M(x, \Sigma) \leq r \right\}.
\end{equation*}

\noindent and {\em the tubular hypersurface} of radius $r$ about $\Sigma$ in $M$ is the set
\begin{equation*}
  \partial T_r(\Sigma):=\left\{ x \in M\, :\, \text{dist}_M(x, \Sigma) = r \right\}.
\end{equation*}
\end{definition}
\begin{remark}
 If $M$ is complete and $\Sigma$ is closed then $C(\Sigma) >0$. Thus  we have  for all  $0\leq r\leq C(\Sigma)$ that
\begin{equation*}
T_r(\Sigma)=\exp^{\perp}\Big(\left\{(p, \xi)\in T\Sigma^\bot, \text{ with }\, \Vert \xi\Vert \leq r\right\}\Big) \subseteq  \exp^{\perp}\left(\Omega_\Sigma\right).
\end{equation*}
\noindent so, as $\Sigma$ is closed, the points in the tube of radius $r$ are those whose distance to the submanifold is less than or equal to $r$, that is, those that can be reached by a minimizing geodesic that starts from the submanifold orthogonally. 

\noindent On the other hand, given $(p, \xi)\in ST\Sigma^\bot$, and given $r \in (0, C(\Sigma))$, let $\gamma_\xi$ be the unique geodesic, unit speed and perpendicular to $\Sigma$ from $p \in \Sigma$ to $\gamma_\xi(r)$ which realizes the distance $dist_M(\gamma_\xi(r),\Sigma)$. Then, the tangent space in $\gamma_\xi(r)$ to the tubular hypersurface $ \partial T_r(\Sigma)$ is given by 
$$T_{\gamma_\xi(r)} \partial T_r(\Sigma)=\{\gamma_\xi'(r)\}^\bot$$
\noindent where $\gamma_\xi'(r)=\grad^M{\rm dist_M}( \cdot, \Sigma)$
\end{remark}

Given $(p, \xi) \in ST\Sigma^{\perp}$, let $$S_{\xi}(t)\colon \{\gamma_{\xi}'(r)\}^{\perp} \to \{\gamma_{\xi}'(t)\}^{\perp}\,\,\,{\rm and}\,\,\,\, R_{\xi}(r) \colon \{\gamma_{\xi}'(r)\}^{\perp} \to \{\gamma_{\xi}'(r)\}^{\perp}$$ be respectively the Shape and the Curvature operators of the hypersurface $\partial \Sigma_{r}$ at $\gamma_{\xi}(r)$ as defined in \cite{gray} using Fermi coordinates. It is known, (see \cite{gray}),  that $S_{\xi}(r)$ and $R_{\xi}(r)$ satisfies the Ricatti equation
 \[ S_{\xi}'(r)=S_{\xi}^{2}(r)+ R_{\xi}(r)\] 
 
\noindent  at every point of $\partial T_r( \Sigma)$ where $S_{\xi}'(r)=\nabla_{\!\!_{\gamma_{\xi}'(r)}}S_{\xi}(r)$.

When $m=n-1$, the eigenvalues of the shape operator $S_{\xi}(r)$, which we denote as $\{k_i(r)\}_{i=1}^{n-1}$, will satisfy the limit condition
\begin{equation}\label{limitcond1}
\lim_{r \to 0} k_i(r)=k_i\,\,\forall i=1,...,n-1\end{equation}
\noindent where $\{k_i\}_{i=1}^{n-1}$ are the eigenvalues of the Weingarten map $L_\xi$ of the hypersurface $\Sigma$.

On the other hand, when $m=0$, namely, $\Sigma$ reduces to a point, then 
\begin{equation}\label{limitcond2}
\lim_{r \to 0} k_i(r)=-\infty\,\,\forall i=1,...,n-1\end{equation}

 
\noindent The mean curvature vector field of the tubular hypersurface $\partial T_r( \Sigma)$ at $\gamma_{\xi}(r)$ with respect to $\gamma_\xi'(r)=\grad^M {\rm dist_M}( \cdot, \Sigma)$ is given by
$$\overline{H}_{\partial T_{r}(\Sigma)}= -{\rm div}^{M}\big(\grad^M {\rm dist_M}( \cdot, \Sigma)\big)\cdot \grad^M {\rm dist_M}( \cdot, \Sigma)$$

 \noindent so the mean curvature pointed outward of the tubular hypersurface $\partial T_r( \Sigma)$ at $\gamma_{\xi}(r)$ with respect to $\gamma_\xi'(r)=\grad^M {\rm dist_M}( \cdot, \Sigma)$ is given by $$H(\gamma_{\xi}(t))=\langle\overline{H}_{\partial T_{r}(\Sigma)}, \grad^M {\rm dist_M}( \cdot, \Sigma)\rangle=\-{\rm tr}\,S_{\xi}(t).$$ 

Let us consider a smooth function $u\colon T_r(\Sigma)\to \mathbb{R}$ defined on the tube of radius $r$ around $\Sigma$ and let $l\colon \Sigma_r \to [0, r]$ be defined by $l(x)={\rm dist_M}(x, \Sigma)$. The Laplacian of $u$ in $T_r(\Sigma)$ is expressed in Fermi coordinates by $$ \triangle_{_M} u (l,p)=\frac{\partial^{2} u}{\partial l^2}(l,p)-{\rm tr}S_{\xi}(l)\frac{\partial u}{\partial l}(l,p)+ \triangle_{_{\partial \Sigma_l}}u(l,p)$$ 
\noindent If $u$ depends only on the distance to $\Sigma$ we say that $u$ is said radial function and 

\begin{equation}\label{eqradial} \triangle_{_M} u (l)=\frac{\partial^{2} u}{\partial l^2}(l)-{\rm tr}S_{\xi}(l)\frac{\partial u}{\partial l}(l).\end{equation}

\section{Proof of Theorem \ref{thmmain}}

In the following Proposition, we give an explicit description of the solutions of the problems \eqref{eqMET2} and \eqref{eqMET3}.

\begin{prop}\label{meanexitmodel}
    The solutions  of the Poisson equations \eqref{eqMET2} and \eqref{eqMET3}, defined, respectively, on  $T_{\delta}(\mathbb{S}^{n-1})$, and on $B_{\pi/2 - \delta}(p_N)$, with $0< \delta < \pi/2$, are given by the following radial functions, depending, respectively, on $s=s(x)={\rm dist}_{\mathbb{S}^{n}}( x, \mathbb{S}^{n-1})$ and on $t(x)={\rm dist}_{\mathbb{S}^{n}}(x, p_N)$.

     \begin{equation}\label{solutionmodel1}
       u_{\delta}(s)=\int_{s}^{\delta}\frac{\int_{0}^{\tau} \cos^{n-1}(l)dl}{\cos^{n-1}(\tau)}d\tau,\,\,\, s\in [0, \delta] 
    \end{equation}
    
    \begin{equation}\label{solutionmodel2}
        v_{\delta}(t)=\int_{t}^{\pi/2-\delta} \frac{\int_{0}^{\tau} \sin^{n-1}(l)dl}{\sin^{n-1}(\tau)}d\tau,\,\,\, t\in [0, \pi/2-\delta] 
    \end{equation}
    
\end{prop}
\begin{proof}
We can find the computation of $v_\delta(t)$ in \cite{HMP12}. For the sake of completeness, we will present the details of the computations of both functions. Let us start with $u_\delta$.  We have to show that \begin{equation} \label{eqL}  
\triangle_{_{\mathbb{S}^{n}}}u_{\delta} (s)=\frac{\partial^{2} u_{\delta}}{\partial s^2}(s)-{\rm tr}\,S(s)\frac{\partial u_{\delta}}{\partial s}(s)=-1. 
\end{equation} 

The Shape operator $S_{\xi}(s)$ satisfies the equation \begin{equation}\label{eqRicatti} S_{\xi}'(s)=S_{\xi}^2(s)+R_{\xi}(s)\end{equation}. In the sphere $S^n$ the curvature operator $R_{\xi}={\rm Id}_{n-1}$. Let $\gamma_{\xi}$ be the unit geodesic normal to $\Sigma$ at $\gamma_{\xi}(0)=p$, $\gamma_{\xi}'(0)=\xi$ and let $\{ e_1(s), \ldots, e_{n-1}(s)\}$ be a parallel  ortonormal  basis along $\gamma_{\xi}$ that  diagonalizes the shape operator $S_{\xi}(s)$, this is \begin{equation}S_{\xi}(s)e_i(s)=k_{i}(s) e_{i}(s)\label{eqS}\end{equation} for $i=1, \ldots n-1.$ Here $k_i(s)$ are the principal curvatures of the tubular hypersurface $\partial T_s(\mathbb{S}^{n-1})$, and, as we have explained in Section \ref{tubes}, equation \eqref{limitcond1}, $k_i(0)=0\,\,\forall i=1,...,n-1$, because $\mathbb{S}^{n-1}$ is totally geodesic in $\mathbb{S}^{n}$.

Hence, from \eqref{eqRicatti}, \eqref{eqS} and the fact that the immersion is totally geodesic, we have, (see \cite[Cor. 3.5]{gray}),

$$\left\{\begin{array}{lll}k_{i}'(s)&=&k_{i}^{2}(s)+1\\
k_i(0)&=&0\end{array}\right.,$$ 
\noindent  Hence, $${\rm tr}\, S_{\xi}(s)=(n-1) \tan s$$

 The equation \eqref{eqL} becomes $$\triangle_{\mathbb{S}^{n}}u_{\delta} (s)=\frac{\partial^{2} u_{\delta}}{\partial s^2}(s)-(n-1) \tan s\,\frac{\partial u_{\delta}}{\partial s}(s)=-1 $$ 
 \noindent and clearly the function $u_{\delta}$ in \eqref{solutionmodel1} satisfies $\triangle_{\mathbb{S}^{n}}u_{\delta} (s) =-1$ with $u_{\delta} (\delta)=0$. 
 
 On the other hand, to compute $v_\delta$, we must have into account now that, as the submanifold we are considering now  is the point $p_N$, the level hypersurfaces of the distance function ${\rm dist}_{\mathbb{S}^{n}}(., p_N)$ are now geodesic spheres in $\mathbb{S}^{n}$ , so, when we take the unit normal pointing to the center $p_N$,  the eigenvalues of the shape operator satisfy the equations, (see again \cite[Cor. 3.5]{gray} and equation \eqref{limitcond2} in section \ref{tubes}), 
 
 $$\left\{\begin{array}{lll}k_{i}'(t)&=&k_{i}^{2}(t)+1\\
k_i(0)&=&-\infty\end{array}\right.,$$

\noindent Hence, $${\rm tr}\, S_{\xi}(t)=-(n-1) \cot t$$
 
 \noindent so the solution $v_\delta(t)$ of \eqref{eqMET3} is given, doing the same computations than above, by $$v_{\delta}(t_x)=\int_{t}^{\pi/2-\delta} \frac{\int_{0}^{\tau} \sin^{n-1}(s)ds}{\sin^{n-1}(\tau)}d\tau$$ 
 \end{proof}
 
 Now, to finish the proof of the Theorem, we are going to establish the behavior of these functions when the dimension increases.
 
 The function $\cos^{n-1}(s)$ is decreasing for $s\in (0, \pi/2)$, so therefore, for each $s=s(x)={\rm dist}(x, \mathbb{S}^{n-1})$, with $x \in T_\delta(\mathbb{S}^{n-1})$,
\begin{eqnarray} u_{\delta} (s)&=&\int_{s}^{\delta}\frac{\int_{0}^{\tau}\cos^{n-1}(l)dl}{\cos^{n-1}(\tau)}d\tau\nonumber \\ &\geq & \frac{1}{\cos^{n-1}(s)}\int_{s}^{\delta}\int_{0}^{\tau}\cos^{n-1}(l)dld\tau\nonumber \\
&\geq & \frac{1}{\cos^{n-1}(s)\sin (\delta)}\int_{s}^{\delta}\int_{0}^{\tau}\cos^{n-1}(l)\sin(l)dld\tau\nonumber \\
&=&\frac{1}{\cos^{n-1}(s)\sin (\delta)}\int_{s}^{\delta}\frac{-\cos^{n}(l)}{n}\vert_{0}^{\tau}d\tau\nonumber \\
&=& \frac{1}{\cos^{n-1}(s)\sin (\delta)}\int_{s}^{\delta}\frac{1-\cos^{n}(\tau)}{n}d\tau\nonumber \\
&\geq & \frac{1}{\cos^{n-1}(s)\sin (\delta)}\frac{1-\cos^{n}(s)}{n}(\delta-s)\nonumber =F_{\delta, n}(s)
\end{eqnarray}

Now, we note that, given $s \in (0,\pi/2)$, we have 
 $$\lim_{n \to \infty}n\cos^{n-1}(s)=\lim_{n \to \infty} n \cos^{n}(s) =\lim_{n \to \infty} e^{ n(\frac{\ln n}{n}+\ln(\cos s))}= 0$$ 
 \noindent To see the last equality, let us observe that, as $0 < \cos(s) < 1\,\,\forall s \in (0, \pi/2)$, then $\ln(\cos(s))<0\,\,\forall s \in (0, \pi/2)$.
 
 Hence, as, on the other hand, $\lim_{n \to \infty} \frac{\ln n}{n}=0$, then, for all $s \in (0,\pi/2)$, 
 $$\lim_{n \to \infty} \frac{\ln n}{n}+\ln(\cos(s))<0$$
 \noindent so
 $$\lim_{n \to \infty} e^{n(\frac{\ln n}{n}+\ln(\cos s))}= 0$$ 

\noindent Therefore, we conclude that
\begin{equation}\nonumber
\begin{aligned}
\lim_{n \to \infty} F_{\delta,n}(s)&=\lim_{n\to \infty}\frac{(\delta-s)}{\cos^{n-1}(s)\sin (\delta)}\frac{1-\cos^{n}(s)}{n}\\&=\frac{(\delta -s)}{\sin \delta}\lim_{n\to \infty}\frac{1-\cos^{n}(s)}{n\cos^{n-1}(s)}
= \infty
\end{aligned}
\end{equation}

\medskip

 On the other hand, when we consider $v_\delta(t(x))$, the mean exit time form the geodesic ball $B_{\pi/2 - \delta}(p_N)$, we have that, denoting as $t=t(x)$ with $x \in B_{\pi/2 - \delta}(p_N)$,

\begin{equation}\nonumber
\begin{aligned}
 v_{\delta}(t)&= \int_{t}^{\pi/2-\delta} \frac{\int_{0}^{\tau} \sin^{n-1}(l)dl}{\sin^{n-1}(\tau)}d\tau \\& \leq\nonumber  \frac{1}{\cos(\pi/2 - \delta)}  \int_{t}^{\pi/2-\delta} \frac{\int_{0}^{\tau} \sin^{n-1}(l)\cos (l)dl}{\sin^{n-1}(\tau)}d\tau
  \\&= \frac{1}{n\cos(\pi/2 - \delta)} (\cos t-\cos(\pi/2-\delta))\nonumber \\& \leq \frac{\cos t}{n\cos(\pi/2 - \delta)} (1-\cos(\pi/2-\delta)) 
\end{aligned}
\end{equation}

Therefore, for all $x \in B_{\pi/2 - \delta}(p_N)$ with $t(x)={\rm dist}_{\mathbb{S}^{n}}(x, p_N) \leq \pi/2-\delta$,  we obtain
$$\lim_{n\to \infty} v_{\delta}(t(x))=0.$$

\subsection{Proof of Corollary \ref{thmmain2}}\

The $k^{\rm th}$-exit moments $u_{k,\Omega,n}(x)$ of a bounded domain $\Omega$ with boundary $\partial \Omega \neq \emptyset$ satisfies the following recurrent family of  Poisson problems, where  $u_{0,\Omega,n}=1$, 
\begin{equation}\label{eqMoment}
    \left\{\begin{array}{rl}
     \triangle^M u_{k,\Omega,n}+ k u_{k-1,\Omega, n} = 0    & {\rm in}\,\,\,\, \Omega  \\
         u_{k,\Omega,n}=0& {\rm on}\,\,\partial \Omega 
    \end{array}\right.
\end{equation}

 If $\Omega=T_\delta(\mathbb{S}^{n-1})\subset \mathbb{S}^{n}$, it is straightforward to check that the solution of  \eqref{eqMoment} is given by
  \[ u_{\delta,k,n}(x)=k\int_{s(x)}^{\delta}\frac{1}{\cos^{n-1}(\tau)}\int_{0}^{\tau}\cos^{n-1}(\xi)u_{\delta,k-1,n}(\xi)d\xi d\tau\] 
and if $\Omega =B_{(\pi/2-\delta)}(p_N)$, the solution of  \eqref{eqMoment} is given by, (see \cite{HMP12})
\[ v_{\delta,k,n}(x)=k\int_{s(x)}^{\pi/2-\delta}\frac{1}{\sin^{n-1}(\tau)}\int_{0}^{\tau}\sin^{n-1}(\xi)u_{\delta,k-1,n}(\xi)d\xi d\tau\]

We are going to prove assertion (i) in Corollary \ref{thmmain2} by induction on $k$, having into account that, from its definition, it is easy to check that the functions $u_{\delta,k,n}(s)$ are non-negative and non-increasing for all $k \geq 1$. For $k=1$, Theorem \eqref{thmmain} states that the first exit moment $u_{\delta,1,n}(x)=u_{\delta,1,n}(s(x))$ is bounded from below by the composition $$(F_{\delta,1,n}\circ s)(x)=F_{\delta,n}(s(x))= \frac{(\delta-s(x))}{\sin(\delta)}\cdot\frac{1-\cos^{n}(s(x))}{n\cos^{n-1}(s(x))}\,\,\,\,\forall x \in T_{\delta}(\mathbb{S}^{n-1})$$
\noindent where $F_{\delta,n}: (0,\delta) \rightarrow \mathbb{R}$ is the function defined as
$$F_{\delta,n}(r)=\frac{(\delta-r)}{\sin(\delta)}\cdot\frac{1-\cos^{n}(r)}{n\cos^{n-1}(r)}\raisepunct{.} $$
\noindent which satisfies  that
$$ \lim_{n\to \infty}F_{\delta,n}(r)=\infty \,\,\,\forall r \in (0,\delta)$$

Suppose that there exists a function $F_{\delta,k-1,n}: (0,\delta) \rightarrow \mathbb{R}$ such that $$u_{\delta,k-1,n}(s(x)) \geq F_{\delta,k-1,n}(s(x)) \,\,\,\,\forall x \in T_{\delta}(\mathbb{S}^{n-1})$$ and which satisfies  that
$$ \lim_{n\to \infty}F_{\delta,k-1,n}(r)=\infty \,\,\,\forall r \in (0,\delta)$$

Fix $x\in T_\delta(\mathbb{S}^{n-1})$. As $u_{\delta,k-1,n}(s)$ is non-negative and non-increasing, we have, taking $\epsilon(x)=\frac{\delta-s(x)}{2} >0$, (so $s(x)<s(x)+\epsilon(x)=\frac{s(x)+\delta}{2} < \delta$), that 
$$u_{\delta,k-1,n}(\xi) \geq u_{\delta,k-1,n}(\tau)\,\,\forall \xi \in (0,\tau)$$
\noindent and 
$$u_{\delta,k-1,n}(\tau) \geq u_{\delta,k-1,n}(\frac{s(x)+\delta}{2}) \,\,\forall \tau \in (s(x),\frac{s(x)+\delta}{2}),$$ 
\noindent so we conclude, using the induction hypothesis, th fact that the functions $u_{\delta,k,n}$ are non-negative and non-increasing for all $k \geq 1$ and inequalities above:

\begin{eqnarray} u_{\delta,k,n}(x) &= & k\int_{s(x)}^{\delta}\frac{1}{\cos^{n-1}(\tau)}\int_{0}^{\tau}\cos^{n-1}(\xi)u_{\delta,k-1,n}(\xi)d\xi d\tau\nonumber \\
&& \nonumber \\
&\geq& k\int_{s(x)}^{\frac{s(x)+\delta}{2}}\frac{1}{\cos^{n-1}(\tau)}\int_{0}^{\tau}\cos^{n-1}(\xi)u_{\delta,k-1,n}(\xi)d\xi d\tau \nonumber  \\
&& \nonumber \\
&\geq & k \cdot u_{\delta,k-1,n}(\frac{s(x)+\delta}{2})\cdot \frac{\delta-s(x)}{2\cos^{n-1}(s(x))\sin(\frac{s(x)+\delta}{2})}\cdot \frac{1-\cos^{n}(s(x))}{n}\nonumber \\
&\geq & k \cdot F_{\delta,k-1,n}(\frac{s(x)+\delta}{2})\cdot \frac{\delta-s(x)}{2\cos^{n-1}(s(x))\sin(\frac{s(x)+\delta}{2})}\cdot \frac{1-\cos^{n}(s(x))}{n}\nonumber 
\end{eqnarray}

Hence, if we define, given $r \in (0,\delta)$:
$$F_{\delta,k,n}(r):=k \cdot F_{\delta,k-1,n}(\frac{r+\delta}{2})\cdot \frac{\delta-r}{2\cos^{n-1}(r)\sin(\frac{r+\delta}{2})}\cdot \frac{1-\cos^{n}(r)}{n}$$
\noindent we have proved that, for all $x \in T_\delta(\mathbb{S}^{n-1})$,
$$u_{\delta,k,n}(s(x)) \geq F_{\delta,k,n}(s(x))$$
\noindent and, moreover, we have that 
$$\lim_{n\to \infty}F_{\delta,k,n}(r)=\infty \,\,\,\forall r \in (0,\delta)$$
\noindent To see this last conclusion, we argue as follows:
Given $0<r <\delta<\pi/2$, since  $\lim_{n \to \infty}(n\cos^{n-1}(r))=0$,  then, in an analogous way than in the proof of Theorem \ref{thmmain}, we have that  
$$\lim_{n\to \infty} k \cdot \frac{\delta-r}{2\cos^{n-1}(r)\sin(\frac{r+\delta}{2})}\cdot \frac{1-\cos^{n}(r)}{n}= \infty$$
\noindent As, on the other hand, we have, by induction hypothesis, that, for all $r \in (0,\delta)$,
$$\lim_{n \to \infty} F_{\delta,k-1,n}(\frac{r+\delta}{2})=\infty$$
\noindent  we get the desired equality.

\noindent The proof of the bound $$ v_{\delta,k,n}(s(x)) \leq k! \left(v_{\delta,1,n}(0)\right)^k $$ \noindent is straightforward and it can be found in p. 6 of  \cite{bmp}. Moreover, it is evident that $\lim_{n \to \infty} v_{\delta,1,n}(s(x))=0\,\,\forall x \in B_{(\pi/2-\delta)}(N)$.

\section{Proof of Theorem  \ref{thmmain3}}

Let $\varphi \colon \Sigma \to M$ a closed embedded minimal hypersurface, where $M$ is a closed Riemannian $n$-manifold with Ricci curvature ${\rm Ric}_{M}\geq (n-1)$. Let $T_\delta(\Sigma)$, with $\delta \in (0, {\rm min}\{\pi/2, C(\Sigma)\})$, be the $\delta$-tubular neighbourhood around $\Sigma$ and $w_{\delta}(x)=\mathbb{E}[\tau_{_{\Sigma_{\delta}}}^x]$ the mean exit time function defined on $T_\delta(\Sigma)$. We know that $w_\delta$ is the solution of the Poisson's problem

\begin{equation}\label{eqMET4}
    \left\{\begin{array}{rl}
         \triangle_{_M} w_{\delta}+1=0& \!{\rm in}\,\,T_{\delta}(\Sigma)  \\
         w_{\delta}=0&\!\!{\rm on} \, \,\partial T_{\delta}(\Sigma).
    \end{array}\right.
\end{equation}

Taking Fermi coordinates $(s, p)$, where $s= s_M(x)={\rm dist}_{M}(x, \Sigma)$ in the tube  $T_\delta(\Sigma)$, we have that the Laplacian of $w_{\delta}$ in $T_\delta(\Sigma)$ is expressed by 
$$ \triangle_{_M} w_{\delta} (s,p)=\frac{\partial^{2} w_\delta}{\partial s^2}(s,p)-{\rm tr}\widetilde{S}_{\xi}(s)\frac{\partial w_\delta}{\partial s}(s,p)+ \triangle_{\partial T_s(\Sigma}w_\delta(s,p)$$
\noindent where the mean curvature pointed outward of the tubular neighborhood $\partial T_r( \Sigma)$ at $\gamma_{\xi}(s)$ with respect to $\grad^M {\rm dist}_M( \cdot, \Sigma)$ is given by $$H(\gamma_{\xi}(s))=\langle \overline{H}_{\partial T_{s}(\Sigma)}, \grad^M {\rm dist}_M( \cdot, \Sigma)\rangle={\rm tr}\,\tilde{S}_{\xi}(s).$$ 

Let us consider now the tube $T_\delta(\mathbb{S}^{n-1})$ around the totally geodesic $\mathbb{S}^{n-1}$ in the sphere $\mathbb{S}^{n}$ and $u_{\delta}$ be its mean exit time function. We have seen in the proof of Theorem \eqref{thmmain} that the function $u_{\delta}$ depends only on the distance $s={\rm dist}_{\mathbb{S}^{n}}(\cdot, \mathbb{S}^{n-1})$ and is given by equation \eqref{solutionmodel1}.

Transplant $u_{\delta}$ to $T_\delta(\Sigma)$ by defining $\tilde{u}_{\delta}\colon T_\delta(\Sigma) \to [0, \infty)$ as $\tilde{u}_{\delta}(x)= u_{\delta}(s_M(x))$ for all $x\in T_\delta(\Sigma)$, where $s_M(x)={\rm dist}_{M}(x, \Sigma).$ Note that $\frac{d}{ds}\tilde{u}_{\delta}= \frac{d}{ds}u_{\delta}$ and $\frac{d^2}{ds^2}\tilde{u}_{\delta}= \frac{d^2}{ds^2}u_{\delta}$ We calculate now $\triangle_{_M}\tilde{u}_{\delta}$, having into account that $\tilde{u}_{\delta}$ is a radial function,  using equation \eqref{eqradial}, (with the parameter $s$), so we obtain

\begin{equation}
\begin{aligned}
\triangle_{_M}\tilde{u}_{\delta}(s_M(x))&= \frac{\partial^{2}\tilde{u}_{\delta}}{\partial s^2}(s)-{\rm tr}\tilde{S}_{\xi}(s)\frac{\partial \tilde{u}_{\delta}}{\partial s}(s)\\
&= \frac{\partial^{2} u_{\delta}}{\partial s^2}(s)-{\rm tr}\tilde{S}_{\xi}(s)\frac{\partial u_{\delta }}{\partial s}(s)\\
&= \frac{\partial^{2} u_{\delta}}{\partial s^2}(s)-(n-1)\tan (s)\frac{\partial u_{\delta}}{\partial s}(s) \\&+ (n-1)\tan (s)\frac{\partial u_{\delta }}{\partial s}(s) -{\rm tr} \tilde{S}_{\xi}(s)\frac{\partial u_{\delta }}{\partial s}(s)\\
&= \triangle_{\mathbb{S}^{n}} u_{\delta}(t)+\left[ (n-1)\tan (t)-{\rm tr} \tilde{S}_{\xi}(t)\right]\frac{\partial u_{\delta }}{\partial s}(s)\\
&= -1 + \left[ (n-1)\tan (t)-{\rm tr} \tilde{S}_{\xi}(t)\right]\frac{\partial u_{\delta }}{\partial s}(s)\\
&=\triangle_{_M}w_{\delta}(x) + \left[ (n-1)\tan (t)-{\rm tr} \tilde{S}_{\xi}(t)\right]\frac{\partial u_{\delta }}{\partial s}(s)
\end{aligned}
\end{equation} 
Therefore $\triangle_{_M}\left(\tilde{u}_{\delta}-w_{\delta}\right)(x)=\left[ (n-1)\tan (t)-{\rm tr} \tilde{S}_{\xi}(t)\right]\displaystyle\frac{\partial u_{\delta }}{\partial s}(s(x))$.
 We can check directly that $\displaystyle\frac{\partial u_{\delta} }{\partial s}(s)\leq 0,\,\,\forall s$ and, by \cite[lemma 8.28 iii.]{gray} that $$(n-1)\tan (s)-{\rm tr}\tilde{S}_{\xi}(s)\leq 0\,\,\forall s.$$ Thus $\triangle_{_M}\tilde{u}_{\delta}(x)\geq -1=\triangle_{_M}w_{\delta}(x)\,\,\forall x \in T_\delta(\Sigma)$. Since $\tilde{u}_{\delta}(x)=w_{\delta}(x)=0$ for $x \in \partial T_\delta(\Sigma)$ this implies, by the Maximum principle, that $\tilde{u}_{\delta}(x)\leq w_{\delta}(x)$ for all $x\in T_\delta(\Sigma)$. 
 \medskip

To discuss the equality case we procceed as follows. If $ \tilde{u}_{\delta}(x_0)=w_{\delta}(x_0)$, since 
 the function $h(x)=\tilde{u}_{\delta}(x)- w_{\delta}(x)\leq 0$ is subharmonic with $h(x_0)=0$ we have  $h\equiv 0$ or equivalently
 $(n-1)\tan (r)={\rm tr}\tilde{S}(r)$ for all $r\in [0, \delta)$. Therefore, we have that $\partial T_r(\Sigma)$ has constant mean curvature $\tan r$ for all $r \in (0,\delta)$ and we have proved assertion (1).
 On the other hand, as 
 $$
\frac{d}{dr}{\rm vol}(T_r(\Sigma))={\rm vol}(\partial T_r(\Sigma)),\quad \frac{d^2}{dr^2}{\rm vol}(T_r(\Sigma))=-\int_{\partial T_r(\Sigma)}{\rm tr}\tilde{S}(r)dV
 $$
we conclude that
$$
\begin{aligned}
    \frac{d^2}{dr^2}{\rm vol}(T_r(\Sigma))&=-(n-1)\tan(r)\frac{d}{dr}{\rm vol}(T_r(\Sigma))\\
    &=\frac{\frac{d^2}{dr^2}{\rm vol}(T_r(\mathbb{S}^{n-1}))}{\frac{d}{dr}{\rm vol}(T_r(\mathbb{S}^{n-1}))}\frac{d}{dr}{\rm vol}(T_r(\Sigma)).
\end{aligned}
$$
This implies that the function 
$$
r\mapsto \frac{\frac{d}{dr}{\rm vol}(T_r(\Sigma))}{\frac{d}{dr}{\rm vol}(T_r(\mathbb{S}^{n-1}))}=\frac{{\rm vol}(\partial T_r(\Sigma))}{{\rm vol}(\partial T_r( \mathbb{S}^{n-1}))}=F(r)
$$
is a constant function, so we have, as $\partial T_r(\Sigma)=\{x \in M/{\rm dist}_M(x,\Sigma)=r\}$ and $\partial T_r(\mathbb{S}^{n-1})=\{\tilde{x} \in \mathbb{S}^{n}/{\rm dist}_{\mathbb{S}^{n}}(x,\mathbb{S}^{n-1})=r\}$, that $\partial T_0(\Sigma)=\Sigma$ and $\partial T_0(\mathbb{S}^{n-1})=\mathbb{S}^{n-1}$, and hence
$$\lim_{r \to 0}F(r)= \frac{{\rm vol}(\Sigma)}{{\rm vol}(\mathbb{S}^{n-1})}$$

\noindent  Therefore
$$
\frac{d}{dr}{\rm vol}(T_r(\Sigma))=\frac{{\rm vol}(\Sigma)}{{\rm vol}(\mathbb{S}^{n-1})}\frac{d}{dr}{\rm vol}(T_r(\mathbb{S}^{n-1}))
$$
By integration and by Bishop-Gromov volume comparison theorem we can state that
\begin{equation}\label{eq:19}
{\rm vol}(\mathbb{S}^n)\geq {\rm vol(M)}\geq {\rm vol}(T_\delta(\Sigma))=\frac{{\rm vol}(\Sigma)}{{\rm vol}(\mathbb{S}^{n-1})}{\rm vol}(T_\delta(\mathbb{S}^{n-1})).    
\end{equation}
 Hence, by using the concentration of measure on the sphere for the tube $T_\delta(\mathbb{S}^{n-1})$, namely, 
$${\rm vol}(T_\delta(\mathbb{S}^{n-1})\geq (1-2e^{-(n-1)\delta^{2}/2} ){\rm vol}(\mathbb{S}^{n}),$$
\noindent we conclude that
\begin{equation}\label{eq:cor}
    {\rm vol}(\mathbb{S}^n)\geq {\rm vol}(M)\geq {\rm vol}(T_\delta(\Sigma))\geq\frac{{\rm vol}(\Sigma)}{{\rm vol}(\mathbb{S}^{n-1})}(1-2e^{-(n-1)\delta^{2}/2} ){\rm vol}(\mathbb{S}^{n}),
\end{equation}
and hence
$$
{\rm vol}(\Sigma)\leq \frac{1}{1-2e^{-(n-1)\delta^{2}/2}}\cdot {\rm vol}(\mathbb{S}^{n-1}).
$$
\subsection{Proof of Corollary \ref{cor:moments}}\

    The $k^{\rm th}$-exit moment functions $u_{\delta,k,n}(\tilde{x})=u_{\delta,k,n}(s_{\mathbb{S}^{n}}(\tilde{x}))$, $k\geq 1$, are radial and satisfy the family of  Poisson problems on $T_{\delta}(\mathbb{S}^{n-1})$, $u_{\delta,0,n}=1$, 
    
    \begin{equation}\left\{\begin{array}{rl}\triangle_{\mathbb{S}^{n}}u_{\delta}^{k}+k u_{\delta}^{k-1}=0&{\rm in}\,\, T_{\delta}(\mathbb{S}^{n-1})\\
    &\\
    u_{\delta}^{k}=0& {\rm on}\,\,\, \partial T_{\delta}(\mathbb{S}^{n-1})
    \end{array}\right.\label{eq15}
         \end{equation}

         These problems can be solved explicitly yielding 
         
         $$u_{\delta,k,n}(s(x))=k\int_{s(x)}^{\delta}\frac{1}{\cos^{n-1}(\tau)}\int_{0}^{\tau}\cos^{n-1}(\xi)u_{\delta,k-1,n}(\xi)d\xi d\tau.$$ 
         
         \noindent Observe that 
         
         $$\displaystyle\frac{\partial \, u_{\delta,k,n}}{\partial s}(s)=-k\frac{1}{\cos^{n-1}(s)}\int_{0}^{\tau}\cos^{n-1}(\xi)u_{\delta,k-1,n}(\xi)d\xi \leq 0$$ 
         \noindent since $u_{\delta,k-1,n}\geq 0$, for all $k \geq 1$. Then, the functions $u_{\delta,k,n}$ are non-increasing, for all $k \geq 1$. On the other hand, the $k^{\rm th}$-exit moment functions $w_{\delta,k,n}(x)$, $k\geq 1$, satisfy a family of  Poisson problems on $T_{\delta}(\Sigma)$ as well, with $w_{\delta,0,n}=1$,
         \begin{equation}\left\{\begin{array}{rl}\triangle_{_M}w_{\delta,k,n}+k w_{\delta,k-1,n}=0&{\rm in}\,\, T_{\delta}(\Sigma)\\
    &\\ w_{\delta,k,n}=0& {\rm on}\,\,\, \partial T_{\delta}(\Sigma)
    \end{array}\right.\label{eq16}\end{equation}
    
    \noindent  We are going to prove,  by induction on $k$, that there exists a function $$F_{\delta,k,n}: (0,\delta) \rightarrow \mathbb{R}$$ \noindent such that, if we denote as $$\tilde{u}_{\delta,k,n}:T_{\delta}(\Sigma) \rightarrow \mathbb{R}$$ \noindent  the transplanted function defined as $\tilde{u}_{\delta,k,n}(x):=u_{\delta,k,n}(s(x))\,\,\,\forall x \in T_{\delta}(\Sigma)$, being $s(x)=s_M(x)={\rm dist}_{M}(x, \Sigma)$  then
    
    \begin{equation}\label{eq2'}
    w_{\delta,k,n}(x)\geq \tilde{u}_{\delta,k,n}(s(x))\geq (F_{\delta,k,n}\circ s)(x)\,\,\,\,\forall x \in T_{\delta}(\Sigma).\end{equation}
and moreover, 
$$ \lim_{n\to \infty}F_{\delta,k,n}(r)=\infty \,\,\,\forall r \in (0,\delta), \,\,\,\forall k \in \mathbb{N} $$    

  \noindent We have, first, that Theorem \eqref{thmmain3} implies that induction hypothesis is true for $k=1$, with 
  $$F_{\delta,1,n}(r)=F_{\delta,n}(r)=\frac{(\delta-r)}{\sin(\delta)}\cdot\frac{1-\cos^{n}(r)}{n\cos^{n-1}(r)}$$

  \noindent  Assume that induction hypothesis is true for $k-1$, namely, that $w_{\delta,k-1,n} \geq \tilde{u}_{\delta,k-1,n}$ in $T_{\delta}(\Sigma)$, and that there exists a function $F_{\delta,k-1,n}: (0,\delta) \rightarrow \mathbb{R}$ such that 
  $$w_{\delta,k-1,n}(x)\geq \tilde{u}_{\delta,k-1,n}(s(x)) \geq F_{\delta,k-1,n}(s(x)) \,\,\,\,\forall x \in T_{\delta}(\Sigma)$$ and which satisfies  that
$$ \lim_{n\to \infty}F_{\delta,k-1,n}(r)=\infty \,\,\,\forall r \in (0,\delta)$$  

\noindent Here, $\tilde{u}_{\delta,k-1,n}(x)=u_{\delta,k-1,n}(s_M(x))$ is the transplanted function to $T_{\delta}(\Sigma)$.
    
    Using \eqref{eq16} and \eqref{eq15}, the fact that $\frac{\partial \, \tilde{u}_{\delta,k,n}}{\partial s}(s)=\frac{\partial \, u_{\delta,k,n}}{\partial s}(s)\le 0$, that $\frac{\partial^{2} \, \tilde{u}_{\delta,k,n}}{\partial s^2}(s)=\frac{\partial^{2} \, u_{\delta,k,n}}{\partial s^2}(s)$, the induction hypothesis and \cite[lemma 8.28 iii.]{gray}, we can compute, for the transplanted function $\tilde{u}_{\delta,k,n}: T_{\delta}(\Sigma) \rightarrow \mathbb{R}$ defined as $\tilde{u}_{\delta,k,n}(x)=u_{\delta,k,n}(s_M(x))$:

 \begin{equation}
 \begin{aligned}
 \triangle_{_M}\tilde{u}_{\delta,k,n}&=\displaystyle\frac{\partial^{2}\, \tilde{u}_{\delta,k,n}}{\partial s^{2}}(s)-{\rm tr}\tilde{S}_\xi(s)\frac{\partial \tilde{u}_{\delta,k,n}}{\partial s}(s) \\
&=\displaystyle\frac{\partial^{2}\, \tilde{u}_{\delta,k,n}}{\partial s^{2}}(s)-(n-1)\tan(s)\frac{\partial \tilde{u}_{\delta,k,n}}{\partial s}(s)\\&+(n-1)\tan(s)\frac{\partial \tilde{u}_{\delta,k,n}}{\partial s}(s)-{\rm tr}\tilde{S}_\xi(s)\frac{\partial \tilde{u}_{\delta,k,n}}{\partial s}(s) \\&= \triangle_{_\mathbb{S}^{n}} u_{\delta,k,n}+\big( (n-1)\tan(s)-{\rm tr}\tilde{S}_\xi(s)\big)\frac{\partial \tilde{u}_{\delta,k,n}}{\partial s}(s) \\&= -k u_{\delta,k-1,n}(s)+\big( (n-1)\tan(s)-{\rm tr}\tilde{S}_\xi(s)\big)\frac{\partial \tilde{u}_{\delta,k,n}}{\partial s}(s)\\ &\geq \big((n-1) \tan(s)-{\rm tr}\tilde{S}_\xi(s)\big)\frac{\partial u_{\delta,k,n}}{\partial s}(s) - kw_{\delta,k-1,n} \\
&\geq   - kw_{\delta,k-1,n}\\
&=\triangle_{_M}w_{\delta,k,n}
\end{aligned}
\end{equation}

Therefore $\triangle_{M}(u_{\delta,k,n}-w_{\delta,k,n})\geq 0$. Since $u_{\delta,k,n}-w_{\delta,k,n}=0$ on the boundary  $\partial T_{\delta}(\Sigma)$ one has $u_{\delta,k,n}-w_{\delta,k,n}\leq 0$ so we have proved  inequality
$$w_{\delta,k,n}(x)\geq u_{\delta,k,n}(s(x))\,\,\forall x \in T_{\delta}(\Sigma)$$
\noindent On the other hand, by Corollary \ref{thmmain2}, we know that the $k^{\rm th}$-exit moment $u_{\delta,k,n}$ defined on the tube $T_{\delta}(\mathbb{S}^{n-1})$ is bounded from below by the  a function $F_{\delta,k,n}\circ s$ as
$$
\begin{aligned}
u_{\delta,k,n}(\tilde{x})&=u_{\delta,k,n}(s(\tilde{x}))\geq (F_{\delta,k,n}\circ s)(\tilde{x})\,\,\,\,\forall \tilde{x} \in T_{\delta}(\mathbb{S}^{n-1})\raisepunct{.}
\end{aligned} $$
\noindent where $s(\tilde{x})=\mathrm{dist}_{\mathbb{S}^{n}}(\tilde{x},\mathbb{S}^{n-1})$ and  $F_{\delta,k,n}: (0,\delta) \rightarrow \mathbb{R}$ is a function which satisfies  that
$$ \lim_{n\to \infty}F_{\delta,k,n}(r)=\infty \,\,\,\forall r \in (0,\delta), \,\,\,\forall k \in \mathbb{N} $$

Hence, we can say exactly the same, with same function $F_{\delta,k,n}\circ s$, replacing the distance function $s(\tilde{x})=\mathrm{dist}_{\mathbb{S}^{n}}(\tilde{x},\mathbb{S}^{n-1})$ by the distance in $M$, $s(x)=\mathrm{dist}_{M}(x,\Sigma)$, because if we consider the transplanted function $u_{\delta,k,n}\circ s_M$, then, given $x \in T_{\delta}(\Sigma)$, we have, taking $\tilde{x} \in T_{\delta}(\mathbb{S}^{n-1})$ such that $\mathrm{dist}_{\mathbb{S}^{n}}(\tilde{x},\mathbb{S}^{n-1})=\mathrm{dist}_{M}(x,\Sigma)$, that
$$u_{\delta,k,n}(s_M(x))=u_{\delta,k,n}(s_{\mathbb{S}^{n}}(\tilde{x}))\geq F_{\delta,k,n}(s_{\mathbb{S}^{n}}(\tilde{x}))=F_{\delta,k,n}(s_M(x))$$

\subsection{Proof of  Corollary \ref{teo:Shiohama}}\

Since $K_M\geq 1$ implies   that  ${\rm Ric}_M\geq n-1$, the first part of the corollary follows from theorem \ref{thmmain3}. 
To prove that $M$ is homeomorphic to $\mathbb{S}^n$, 
we only need to prove that the diameter  of $M$ is greater than $\pi/2$. Suppose otherwise that  the diameter of $M$ is bounded from above by ${\rm diam}(M)< \pi/2$, hence by Bishop-Gromov volume comparison theorem and equation  \eqref{eq:cor} we have \[1\geq \frac{{\rm vol}(B_M(D))}{{\rm vol}(B_{\mathbb{S}^{n}}(D))}\geq \frac{{\rm vol}(M)}{{\rm vol}(B_{\mathbb{S}^{n}}(\pi/2))}\geq \frac{{\rm vol}(\Sigma)}{{\rm vol}(\mathbb{S}^{n-1})}2(1-2e^{-(n-1)\delta^{2}/2} ).\] 
This would imply that 
$$
{\rm vol}(\Sigma)\leq \frac{1}{2}\frac{1}{1-2e^{-(n-1)\delta^{2}/2}}{\rm vol}(\mathbb{S}^{n-1}).
$$
Therefore, since we are assuming 
$$
{\rm vol}(\Sigma)> \frac{1}{2}\frac{1}{1-2e^{-(n-1)\delta^{2}/2}}{\rm vol}(\mathbb{S}^{n-1}),
$$
then ${\rm diam}(M)\geq \pi/2$ and the result follows from the celebrated Grove-Shiohama sphere theorem \eqref{thmGS}:

\begin{theorem}[Grove-Shiohama \cite{grove-shiohama77}] Let $M$ be a compact Riemannian $n$-manifold with sectional curvature $K_{M}\geq 1$. If ${\rm diam}(M)\geq \pi/2$ then is homeomorphic to $ \mathbb{S}^{n}$.\label{thmGS}
\end{theorem}




\begin{thebibliography}{1}

\bibitem{banuellos-carroll94} R. Ba\~{n}uelos,  T. Carroll, {\em Brownian motion and the fundamental frequency of a drum.}
Duke Math. J. \textbf{75} (1994), no. 3, 575--602.
\bibitem{banuellos-carroll11} R. Ba\~{n}uelos,  T. Carroll, {\em  The maximal expected lifetime of Brownian motion.} Math. Proc. R. Ir. Acad. \textbf{111A} (2011), no. 1, 1--11.
\bibitem{banuelos23} R. Bañuelos, P. Mariano, J. Wang, {\em Bounds for exit times of Brownian motion and the first Dirichlet eigenvalue for the Laplacian.} Trans. Amer. Math. Soc. \textbf{376} (2023), no. 8, 5409--5432.
\bibitem{bmp}G. P. Bessa, S. Markvorsen, L. Pessoa, {\em Mean exit times from submanifolds with bounded mean curvature.} To appear in POTA.
\bibitem{bessa-montenegro2009}G. P. Bessa, J. F. Montenegro, {\em 
Mean time exit and isoperimetric inequalities for minimal submanifolds of $N\times \mathbb{R}$.} Bull. Lond. Math. Soc. \textbf{41} (2009), no. 2, 242--252.
\bibitem{Cadeddu15}L.  Cadeddu, S. Gallot, , A. Loi,  {\em Maximizing mean exit-time of the Brownian motion on Riemannian manifolds.} Monatsh. Math. \textbf{176} (2015), no. 4, 551--570.
\bibitem{cc} J. Cheeger, T. Colding, {\em On the structure of spaces with Ricci curvature bounded
below.} I, J. Diff. Geom. \textbf{45}, (1997), 406--480.
 

\bibitem{DM} G. Del Grosso, F.  Marchetti, {\em Principal eigenvalues and exit times for diffusions on manifolds.} Systems Anal. Modelling Simulation \textbf{1} (1984), no. 3, 175--182. 
\bibitem{dynkin} E. B. Dynkin, {\em  Markov Process} Vols. 1 and 2, Grundlehren Math. Wiss. Vols. 121 and 122, Springer, Berlin (1965)
\bibitem{DLD} E. B. Dryden, J. J. Langford, and P. McDonald, \textit{Exit time moments and eigenvalue estimates}, Bull. Lond. Math. Soc. \textbf{49}, (2015), 480-490.

\bibitem{gray}A. Gray, {\em Tubes,} second ed. Progress in Mathematics, Birkh\"{a}user Verlag, Basel, \textbf{221}, (2004). With a preface by Vicent Miquel.
\bibitem{gimeno-palmer21}V. Gimeno, V. Palmer, {\em Parabolicity, Brownian exit time and properness of solitons of the direct and inverse mean curvature flow.} J. Geom. Anal. \textbf{31} (2021), no. 1, 579--618.
\bibitem{GiPa} V. Gimeno, V. Palmer, {\em Fat equator effect and minimality in immersions and submersions of the sphere } Preprint arXiv:2404.19416v1 [math.DG], (2024)
\bibitem{grigoryan}A. Grigor'yan, {\em Analytic and geometric background of recurrence and non-explosion of the Brownian motion on Riemannian manifolds.} Bull. Amer. Math. Soc. \textbf{36}, 2, (1999) 135--249
\bibitem{grove-shiohama77}K. Grove, K. Shiohama, {\em 
A generalized sphere theorem.}
Ann. of Math. (2) \textbf{106} (1977), no. 2, 201--211.
\bibitem{gual-maso05}X.  Gual-Arnau, R.  Masó,  {\em On the mean exit time for compact symmetric spaces.} Acta Math. Sin. (Engl. Ser.) \textbf{21} (2005), no. 3, 555--562. 
\bibitem{Gromov} Mikhael Gromov,
\textit{Filling Riemannian manifolds}, 
\newblock {\em J. Differential Geometry}, 18:1--147, 1983.
\bibitem{Ha}  R.Z. Ha'sminskii,  {\em Probabilistic representation of the solution of some differential equations},
in Proc. 6th All Union Conf. on Theor. Probability and Math. Statist. (Vilnius 1960), 1960. (See
Math. Rev. 3127, v. 32 (1966), p. 632).
\bibitem{HMP09} A. Hurtado, S. Markvorsen and V. Palmer,  {\em Torsional rigidity of submanifolds with controlled geometry.}  Math. Ann. \textbf{344} (2009), no. 3, 511--542.
\bibitem{HMP12} A. Hurtado, S. Markvorsen and V. Palmer, Comparison of exit moment spectra for extrinsic metric balls. Potential Anal. \textbf{36} (2012), no. 1, 137--153.
\bibitem{HMP16} A. Hurtado, S. Markvorsen and V. Palmer, {\em Estimates for the first Dirichlet eigenvalue from exit time moment spectra.} Math. Ann. \textbf{365} (2016) 1603--1632.
\bibitem{irie-marques-neves}K. Irei, F. Marques, A. Neves, {\em Density of minimal hypersurfaces for generic metrics.} Ann. of Math. (2) \textbf{187} (2018), no. 3, 963--972.
\bibitem{karp-pinsky}L. Karp, M. Pinsky, {\em Mean exit time from an extrinsic ball.} Pitman Research Notes in Math. \textbf{150} (1986), 176--186.
\bibitem{markvorsen89}S. Markvorsen, {\em On the mean exit time from a minimal submanifold.} J. Diff. Geom. \textbf{29} (1989), 1--8.
\bibitem{marques-neves}F. Marques, A. Neves, {\em 
Existence of infinitely many minimal hypersurfaces in positive Ricci curvature.} Invent. Math. \textbf{209} (2017), no. 2, 577--616.
\bibitem[Mc2]{Mc2} Patrick McDonald. \newblock Exit times, moment problems and comparison theorems. \newblock {\em Potential Anal.}, \textbf{38}, (2013), 1365--1372.
\bibitem{miquel-palmer92}V. Miquel, V. Palmer, {\em A comparison theorem for the mean exit time from a domain in a K\"{a}hler manifold.} Ann. Global Anal. Geom. \textbf{10} (1992), no. 1, 73--80.
\bibitem{MS} Vitali D. Milman and Gideon Schechtman,
\textit{Asymptotic Theory of Finite Dimensional Normed Spaces, with and Appendix by M. Gromov}, 2nd Printing,
Lecture Notes in Mathematics, \textbf{1200},  Springer Verlag,
Berlin (2001).
\bibitem{palmer98}V. Palmer, {\em Mean exit time from convex hypersurfaces  } Proc. Amer. Math. Soc., \textbf{126} no.7 (1998) 2089--2094.
\bibitem{palmer99}V. Palmer,`{\em Isoperimetric inequalities for extrinsic balls in minimal submanifolds and their applications.} J. London Math. Soc. (2) \textbf{60} (1999), no. 2, 607--616.
\bibitem{perelman} G. Perelman, {\em Manifolds of positive curvature with almost maximal volume. } J. of Amer. Math. Soc. \textbf{7} n.2 (1994) 299--305.

\bibitem{schoen-simon}R. Schoen, L. Simon, {\em Regularity of stable minimal hypersurfaces.}
Comm. Pure Appl. Math. \textbf{34} (1981), no. 6, 741--797. 
\bibitem{song}A. Song, {\em Existence of infinitely many minimal hypersurfaces in closed manifolds.} Ann. of Math. (2) \textbf{197} (2023), no. 3, 859–895.
\end{thebibliography}
\end{document}